\documentclass[a4paper,11pt]{amsart}

\usepackage{amsmath}
\usepackage{amssymb}
\usepackage{amsthm}
\usepackage[all]{xy}

\newtheorem{thm}{Theorem}
\newtheorem*{thmA}{Theorem A}
\newtheorem*{thmB}{Theorem B}

\newtheorem{lemma}{Lemma}
\newtheorem{df}{Definition}
\newtheorem{prop}{Proposition}

\newcommand{\lag}{\mathfrak g}
\newcommand{\lap}{\mathfrak p}

\newcommand{\lasp}{\mathfrak{sp}}
\newcommand{\laosp}{\mathfrak{osp}}

\newcommand{\lagl}{\mathfrak{gl}}
\newcommand{\lasl}{\mathfrak{sl}}

\newcommand{\lat}{\mathfrak{t}}

\newcommand{\mS}{\mathbb S}

\newcommand{\mC}{\mathbb C}

\newcommand{\mR}{\mathbb R}
\newcommand{\mN}{\mathbb N}

\newcommand{\Spin}{{Spin}}

 \newcommand{\cH}{\mathcal{H}}

\newcommand{\cP}{\mathcal{P}}
\newcommand{\cM}{\mathcal{M}}

\newcommand{\Ind}{\operatorname{Ind}}
\newcommand{\lspan}{\operatorname{span}}
\newcommand{\Ker}{\operatorname{Ker}}

\def\pa{\partial}

\begin{document}

\title[Separation of variables]{Separation of variables in the semistable range}

\author{R. L\'avi\v cka}

\address{Charles University, Faculty of Mathematics and Physics, Mathematical Institute\\
Sokolovsk\'a 83, 186 75 Praha, Czech Republic}

\email[R. L\'avi\v cka]{lavicka@karlin.mff.cuni.cz}

\thanks{The support of the grant GACR 17-01171S is gratefully acknowledged.}

\dedicatory{Dedicated to Professor Wolfgang Spr\"o\ss ig}


\begin{abstract}
In this paper, we give an alternative proof of separation of variables for scalar-valued polynomials $P:(\mR^m)^k\to\mC$ in the semistable range $m\geq 2k-1$ for the symmetry given by the orthogonal group $O(m)$. 
It turns out that uniqueness of the decomposition of polynomials into spherical harmonics is equivalent to irreducibility of generalized Verma modules for the Howe dual partner $\lasp(2k)$ generated by spherical harmonics.
We believe that this approach might be applied to the case of spinor-valued polynomials and to other settings as well.  
\end{abstract}


\maketitle


\section{Introduction}

It is well-known that each polynomial $P$ in the Euclidean space $\mR^m$ can be uniquely written as a~finite sum
$$P=H_0+r^2H_1+\cdots+r^{2j}H_j+\cdots$$
where $r^2=x_1^2+\cdots+x_m^2$ for $x=(x_1,\ldots,x_m)\in\mR^m$ and $H_j$ are harmonic polynomials in $\mR^m$, that is, $\Delta H_j=0$ for the Laplace operator $$\Delta=\pa_{x_1}^2+\cdots+\pa_{x_m}^2.$$
In other words, the space $\cP$ of $\mC$-valued polynomials on $\mR^m$ decomposes as
$$\cP=\bigoplus_{n=0}^{\infty} r^{2n} \cH$$ 
where $\cH=\Ker(\Delta) \cap\cP$ is the space of spherical harmonics in $\mR^m$. 
This result is known as separation of variables or the Fischer decomposition.
The underlying symmetry is given by the orthogonal group $O(m)$. 
The invariant operators $\Delta, r^2, h$ generate the Lie algebra $\lasl(2)$ 
where $$h=x_1\pa_{x_1}+\cdots+x_m\pa_{x_m} + m/2$$ is the Euler operator.  
This is the so-called hidden symmetry in this case.
Actually, it is a~simple example of Howe duality for {Howe dual pair} $(O(m),\lasl(2))$. 

For separation of variables and Howe duality in various cases and 
for other symmetry groups, see e.g.\ \cite{CKW, CW, C, Cou,G,Ho1,Ho2,KV,L,LS,LvS,Som1}.

For example, spinor valued polynomials in one variable of $\mR^m$ decompose into monogenic polynomials \cite{DSS}. Let $\mS$ be an irreducible spin representation of the group $Pin(m)$, the double cover of $O(m)$. The spinor space $\mS$ is usually realized inside the complex Clifford algebra $\mC_m$ generated by the generators $e_1,\ldots,e_m$ satisfying the relations $e_i e_j+e_j e_i=-2\delta_{ij}$. The Euclidean space $\mR^m$ is embedded into $\mC_m$ as $$(x_1,\ldots,x_m)\to x_1e_1+\cdots+x_me_m.$$ 
A~polynomial $P:\mR^m\to\mS$ is called monogenic if it satisfies the equation $\pa P=0$ where 
$$\pa:=\sum_{i=1}^m e_i\partial_{x_i}$$
is the Dirac operator in $\mR^m$. 
Then each polynomial $P:\mR^m\to\mS$ has a~unique expression as a~finite sum
$$P=M_0+x M_1+\cdots+x^j M_j+\cdots$$
where $x=x_1e_1+\cdots+x_me_m$ is the vector variable of $\mR^m$ and $M_j$ are monogenic polynomials in $\mR^m$.
Then the space of spinor-valued polynomials on $\mR^m$ decomposes as
$$\cP\otimes\mS=\bigoplus_{n=0}^{\infty} x^{n} \cM$$ 
where $\cM=\Ker(\pa) \cap (\cP\otimes\mS)$ is the space of spherical monogenics in $\mR^m$. 
The symmetry is given by the group $Pin(m)$ and 
the invariant operators $\pa, x$ generate the Lie superalgebra $\laosp(1|2)$. 
This is Howe duality $(Pin(m),\laosp(1|2))$. 

The case of scalar or spinor valued polynomials in more variables is more interesting and involved.

\subsection{Scalar-valued polynomials}

Let us start with the scalar case which has been studied for a~long time and is well-understood.
We consider $\mC$-valued polynomials in $k$ vector variables $x^i$ of $\mR^m$.
Here $x^i=(x^i_1,\ldots,x^i_m)\in\mR^m$ for $i=1,\ldots,k$. We take a~natural action of $O(m)$ on the space $\cP$ of polynomials $P:(\mR^m)^k\to\mC$. Then the invariant operators
$$\Delta_{ij}=\pa_{x^i_1}\pa_{x^j_1}+\cdots+\pa_{x^i_m}\pa_{x^j_m},\ \ r^2_{ij}=x^i_1x^j_1+\cdots+x^i_mx^j_m$$  
$$h_{ij}=x^i_1\pa_{x^j_1}+\cdots+x^i_m\pa_{x^j_m} + (m/2)\delta_{ij},\ \  i,j=1,\ldots,k$$
generate the Lie algebra $\lasp(2k)$, and the mixed Euler operators $h_{ij}$ its subalgebra $\lagl(k)$.
This case is indeed Howe duality $(O(m),\lasp(2k))$.
Spherical harmonics are the polynomials in the kernel of all the mixed laplacians $\Delta_{ij}$. Thus we denote
$$\cH=\Ker(\Delta_{ij}, 1\leq i\leq j\leq k)\cap\cP.$$ Let us remark that $\Delta_{ij}=\Delta_{ji}$ and $r^2_{ij}=r^2_{ji}$.
Theorem A below describes, in a~semistable range $m\geq 2k-1$,  a~decomposition of polynomials into spherical harmonics we can view as a~proper generalization of the harmonic Fischer decomposition to more variables.

\begin{thmA}
If $m\geq 2k-1$, then
$$\cP=\bigoplus_{n} r^{2n} \cH \text{\ \ \ \ with\ \ }r^{2n}=\prod_{1\leq i\leq j\leq k}r_{ij}^{2n_{ij}}$$
where the sum is taken over all $n=\{n_{ij},\ 1\leq i\leq j\leq k\}\subset\mN_0$.
\end{thmA}

This result at least in the stable range $m\geq 2k$ is well-known in invariant theory and theory of Howe duality, see \cite{Ho2}. 
In the next section, we give an alternative proof and extend the result even to the semistable range $m\geq 2k-1$.
The non-stable range is much more complicated and less understood. 

\subsection{Spinor-valued polynomials}

The form of the Fischer decomposition for spinor-valued polynomials  in the semistable range was conjectured by F.~Colombo, F. Sommen, I. Sabadini, D. Struppa in 2004 in the book \cite{CSSS}. Before recalling this, let us introduce some notations.
On the space $\cP\otimes\mS$ of polynomials $P:(\mR^m)^k\to\mS$, there is a~natural action of the group $Pin(m)$.  
In this case, we have $k$ vector variables $x^i\in\mR^m$,
$$x^i=e_1 x^i_1+\cdots+e_m x^i_m$$
and $k$ corresponding Dirac operators $\pa^i$,
$$\pa^i =e_1\pa_{x^i_1}+\cdots+e_m\pa_{x^i_m}$$ 
for $i=1,\ldots,k$. Then the odd invariant operators $x^i,\pa^i$ generate the Lie superalgebra $\laosp(1|2k)$,
and its even part $\lasp(2k)$ is generated by the 'scalar' operators $\Delta_{ij}, r^2_{ij}, h_{ij}$ we know from the scalar case.  
The role of spherical harmonics is played by spherical monogenics, that is, polynomial solutions $P:(\mR^m)^k\to\mS$ of the system of all the Dirac equations $\pa^i P=0$ for $i=1,\ldots,k$. Denote $$\cM=\Ker(\pa^1,\ldots,\pa^k)\cap(\cP\otimes\mS).$$
Finally,  for $J\subset\{1,2,\ldots,k\}$, put $x^J=x^{j_1}\cdots x^{j_r}$ where $J=\{j_1,\cdots,j_r\}$ and $j_1<\cdots<j_r$. 
Here $x^{\emptyset}:=1$. Then we have the following result.

\begin{thmB}
If $m\geq 2k$, then
$$\cP\otimes\mS=\bigoplus_{n, J} r^{2n} x^J \cM$$
where the sum is taken over all $J\subset\{1,\dots,k\}$ and $n=\{n_{ij},\ 1\leq i\leq j\leq k\}\subset\mN_0$.
\end{thmB}

The decomposition of Theorem B was conjectured in \cite{CSSS} but even in the semistable range $m\geq 2k-1$.   
Theorem B (that is, this decomposition only in the stable range)  was recently proved in \cite{LS} using the harmonic Fischer decomposition. For the case of two variables, see \cite{L}.
Since the harmonic Fischer decomposition is now extended to the semistable range there is a~hope that the conjecture could be proved in the full semistable range as well.  

It is known \cite{CKW} that, in $\cP\otimes\mS$, the isotypic components of $\Spin(m)$ form irreducible lowest weight modules for $\laosp(1|2k)$ with lowest weights $(a_1+(m/2),\ldots,a_k+(m/2))$ for integers $a_1\geq \cdots\geq a_k\geq 0$ when the dimension $m$ is even. 
Let us remark that there are not so many known explicit realizations of such modules, see e.g.\ the paraboson Fock space \cite{LSJ}.   
For a~classification of such modules, we refer to \cite{DZ,DS}. 

\section{Proof of Theorem A}

In this section, we prove the harmonic Fischer decomposition in the semistable range. We divide the proof into 3 steps. The first two steps are quite standard. In the last step, to show uniqueness of the decomposition of polynomials we study irreducibility of generalized Verma modules for the Howe dual partner $\lasp(2k)$ generated by spherical harmonics. This approach seems to be very flexible and to work well in other settings. In particular, we believe that, using this approach, Theorem B might be proved in the full semistable range by studying the structure of generalized Verma modules for $\laosp(1|2k)$.

\subsection*{Step 1: Decomposition into spherical harmonics}

First we show that each polynomial can be expressed in  terms of spherical harmonics. This is easy. But as we shall see the question of uniqueness of such an expression is more difficult.

\begin{lemma}\label{l_harmdecomp} (i) We have
$$\cP=\cH\oplus\sum_{1\leq i\leq j\leq k}r_{ij}^2\,\cP$$

\noindent\smallskip
(ii) We have 
\begin{equation}\label{e_decomp}
\cP=\sum_{n} r^{2n} \cH \text{\ \ \ \ with\ \ }r^{2n}=\prod_{1\leq i\leq j\leq k}r_{ij}^{2n_{ij}}
\end{equation}
where the sum is taken over all $n=\{n_{ij},\ 1\leq i\leq j\leq k\}\subset\mN_0$.
\end{lemma}

\begin{proof} 
(i) This is an orthogonal decomposition with respect to the Fischer inner product on the space $\cP$ of polynomials.
Here the Fischer inner product is defined, for $P,Q\in\cP$, by
$$(P,Q)=[\overline{P(\pa)}(Q(x))]_{x=0}$$ 
where $\bar z$ is the complex conjugation of $z\in\mC$ and $P(\pa)$ denotes the constant coefficient differential operator obtained by substituting derivatives $\pa_{x^i_j}$ for the variables ${x^i_j}$.
(ii) Use repeatedly (i).
\end{proof}

It is known that, in general, the sum \eqref{e_decomp}  is {not direct}. In Step 3, we show that this sum is direct in the semistable range.

\subsection*{Step 2: Decomposition of spherical harmonics}

It is easy to see that the space $\cH$ of spherical harmonics is invariant with respect not only to the symmetry group $O(m)$ but also to the Lie algebra $\lagl(k)$ generated by the mixed Euler operators $h_{ij}$. Before describing an irreducible decomposition of $\cH$ under the joint action of $O(m)\times\lagl(k)$ let us recall some notations.

A~partition $a=(a_1,\ldots,a_m)$ of the length at most $m$ is a~non-negative integer sequence $a_1\geq a_2\geq \cdots\geq a_m\geq 0$. 
With a~partition $a$ we often identify the corresponding Young diagram, that is, the array of square boxes arranged in left-justified horizontal rows and with row $i$ having just $a_i$ boxes.  Then there is a one-to-one correspondence between finite-dimensional irreducible representations of $O(m)$ and partitions $a$ satisfying the condition  
\begin{equation}\label{e_ortho}
(a')_1+(a')_2\leq m. 
\end{equation}
Here $a'$ denotes the transpose of the Young diagram $a$ and thus the condition \eqref{e_ortho} means that the sum of the first two columns of the Young diagram $a$ is at most $m$.  
See \cite[Sections 5.2.2, 10.2.4 and 10.2.5]{GW} for details. 

Finite-dimensional irreducible representations of $\lagl(k)$ are indexed by highest weights. In what follows, we use the triangular decomposition
$$\lagl(k)=\lat_-\oplus\lat_0\oplus\lat_+$$ with
$\lat_-=\lspan\{h_{ij}, i<j\}$, $\lat_0=\lspan\{ h_{ij}, i=j\}$, $\lat_+=\lspan\{h_{ij}, i>j\}$.

\begin{thm}\label{t_harm} Under the joint action of $O(m)\times\lagl(k)$, we have an irreducible decomposition
$$
\ \ \cH=\bigoplus_a\;\cH^S_{a}\otimes F_{\tilde a}
$$
where the sum is taken over all partitions $a=(a_1,\ldots,a_k)$ of the length at most $k$ and satisfying the condition \eqref{e_ortho} above, 
$\cH^S_a$ is $O(m)$-irreducible module with the label $a$ and 
$F_{\tilde a}$ is $\lagl(k)$-irreducible module with the highest weight
$\tilde a=(a_1+m/2,\ldots,a_k+m/2)$.
\end{thm}

This theorem is true in general, not only in the semistable or stable range. 
Actually, in the case when $k\geq m$, we can realize each finite-dimensional irreducible representation of $O(m)$ inside the space $\cH$ of spherical harmonics. Indeed, we have 
$$\cH^S_a=\cH\cap\Ker(\lat_-)\cap\cP_a$$ where
$\cP_a$ is the subspace of polynomials of $\cP$ homogeneous in $x^i$ of degree $a_i$ for each $i=1,\ldots,k$ and $a$ satisfies the condition \eqref{e_ortho} above.
Let us remark that polynomials of $\cH^S=\cH\cap\Ker(\lat_-)$ are sometimes called simplicial harmonics.
For a~proof of Theorem \ref{t_harm} and for a~construction of highest weight vectors, see \cite[3.6, pp. 37-40]{Ho2} and cf.\ \cite{CLS}. 

\subsection*{Step 3: Uniqueness of the decomposition}

In the last step, we show uniqueness of the decomposition of polynomials in the semistable range.
By Steps 1 and 2, we know
$$\cP=\sum_{n} r^{2n} \cH \text{\ \ and\ \ }\cH=\bigoplus_a\;\cH^S_{a}\otimes F_{\tilde a}$$
where the sums are taken over all $n=\{n_{ij},\ i\leq j\}\subset\mN_0$ and all partitions $a$ of the length at most $k$ and satisfying the condition \eqref{e_ortho}, respectively.
Then we have
\begin{equation}\label{e_isodecomp}
\cP=\bigoplus_a\;\cH^S_{a}\otimes L_{\tilde a}\text{\ \ with\ \ }L_{\tilde a}=\sum_{n} r^{2n} F_{\tilde a}.
\end{equation}
Here the first sum in \eqref{e_isodecomp} is direct because this is an isotypic decomposition for the group $O(m)$.
Moreover, it is easy to see that $L_{\tilde a}$ is a {lowest weight module} with the lowest weight $\tilde a$ for the Howe dual partner $\lasp(2k)$ generated by the invariant operators. Actually, it is well-known that, using Howe duality ($O(m)$, $\lasp(2k)$), 
the module $L_{\tilde a}$ is even irreducible. But we do not need this fact in our argument.
But what we really need is to observe that $L_{\tilde a}$ is a~quotient of a~generalized Verma module $V_{\tilde a}$ for 
$\lasp(2k)$. In the next section, we introduce generalized Verma modules $V_{\lambda}$ for $\lasp(2k)$ and its parabolic subalgebra suitable for our purposes and find out sufficient conditions on the weight $\lambda$ under which $V_{\lambda}$ is irreducible.
It turns out that uniqueness of the decomposition \eqref{e_decomp} is closely related to the structure of  the modules $L_{\tilde a}$.
Indeed, the sum in $L_{\tilde a}$ of \eqref{e_isodecomp} is {direct} if and only if 
$L_{\tilde a}$ is isomorphic to $V_{\tilde a}$. But we show that, in the semistable range $m\geq 2k-1$, all modules $V_{\tilde a}$ are 
in fact {irreducible} (see Proposition \ref{p_irred} and Example below) and hence
$$\cP=\bigoplus_{n} r^{2n} \cH,$$
which completes the proof of Theorem A.

In the non-stable range, the module $L_{\tilde a}$ is not, in general, isomorphic to $V_{\tilde a}$ but it is just a~unique irreducible quotient of $V_{\tilde a}$. So, even in the non-stable range, the study of the decomposition of polynomials \eqref{e_decomp} is closely related to the structure of the modules $L_{\tilde a}$ and representation theory might help much with this task.

\section{Generalized Verma modules for $\lasp(2k)$}

In this section, we introduce generalized Verma modules for $\lasp(2k)$ we need in Step 3 of the proof of Theorem A.
For an account of generalized Verma modules, we refer to \cite{hum}.

In our setting, the Lie algebra $\lasp(2k)$ is generated by the invariant operators $\Delta_{ij}$, $r^2_{ij}$ and $h_{ij}$. 
We have a~decomposition $\lasp(2k)=\lap_-\oplus\lat\oplus\lap_+$ where
$$\lap_-=\lspan\{ \Delta_{ij},\ 1\leq i\leq j\leq k\},\ \ \lap_+=\lspan\{ r^2_{ij},\ 1\leq i\leq j\leq k\},$$ 
$$\lat=\lspan\{ h_{ij},\ 1\leq i,j\leq k\}.$$
We take a~{parabolic} subalgebra $\lap=\lap_-\oplus\lat$ with its Levi subalgebra $\lat\simeq\lagl(k)$.

\begin{df}
Let $F_{\lambda}$ be a~finite dimensional $\lagl(k)$-irreducible module with the highest weight $\lambda$ such that the action of $\lap_- $ on $F_{\lambda}$ is trivial, that is, $(\lap_-)\cdot F_{\lambda}=0$. Then 
we define {the generalized Verma module} for $\lag=\lasp(2k)$ and its parabolic subalgebra $\lap$ as the induced module
$$V_{\lambda}:=\Ind^{\lag}_{\lap} F_{\lambda}.$$
\end{df}

It is well-known that, at least as vector spaces,
$$V_{\lambda}\simeq\bigoplus_{n} r^{2n} F_{\lambda}.$$
The following proposition \cite[9.12, p.\ 196]{hum} gives sufficient conditions on the weight $\lambda$ under which $V_{\lambda}$ is irreducible.

\begin{prop}\label{p_irred} 
The generalized Verma module $V_{\lambda}$ is irreducible if
\begin{itemize}
\item[(1)] $\lambda_i+\lambda_j-2k+i+j-2\not\in-\mN$ for $1\leq i<j\leq k$, and

\item[(2)] $\lambda_i-k+i-1\not\in-\mN$ for $1\leq i\leq k$.
\end{itemize}
Here $\lambda=(\lambda_1,\ldots,\lambda_k)$.
\end{prop}

\noindent
{\bf Example.} Let $a=(a_1,\ldots,a_k)$ be a~partition and $\tilde a=(a_1+m/2,\ldots,a_k+m/2)$.
Then the conditions of Proposition \ref{p_irred} for $\lambda=\tilde a$ read as
\begin{itemize}
\item[(1')] $a_i+a_j+m-2k+i+j-2\not\in-\mN$ for $1\leq i<j\leq k$, and

\item[(2')] $a_i+(m/2)-k+i-1\not\in-\mN$ for $1\leq i\leq k$.
\end{itemize}
In particular, in the semistable range $m\geq 2k-1$, the conditions (1') and (2') are always satisfied and hence the corresponding module $V_{\tilde a}$ is irreducible.

\subsection*{Acknowledgments}

The author is very grateful for useful advise and suggestions from Vladim\'ir Sou\v cek.

\def\bibname{Bibliography}

\end{document}